\documentclass[11pt,a4paper]{article}
\usepackage{amsmath}
\usepackage{amsthm}
\usepackage{makeidx}
\usepackage{latexsym}
\usepackage{graphicx}
\usepackage{amssymb}
\usepackage{hyperref}
\usepackage{geometry}
\usepackage{stackrel}
\usepackage{bbold}

\newtheorem{proposition}{Proposition}
\newtheorem{corollary}{Corollary}
\newtheorem{remark}{Remark}
\newtheorem{definition}{Definition}
\theoremstyle{remark}
\newtheorem{example}{\textbf{Example}}

\usepackage[latin1]{inputenc}

\title{An isomorphism between the convolution product and the componentwise sum connected to the D'Arcais numbers and the Ramanujan tau function}
\author{Stefano Barbero, Umberto Cerruti, Nadir Murru\\ \\
Department of Mathematics G. Peano, University of Turin,\\
Via Carlo Alberto 10, 10123, Torino, ITALY\\ \\
stefano.barbero@unito.it, umberto.cerruti@unito.it, nadir.murru@unito.it}
\date{}

\begin{document}
\maketitle

\begin{abstract}
Given a commutative ring $R$ with identity, let $H_R$ be the set of sequences of elements in  $R$. We investigate a novel isomorphism between $(H_R, +)$ and $(\tilde H_R,*)$, where $+$ is the componentwise sum, $*$ is the convolution product (or Cauchy product) and $\tilde H_R$ the set of sequences starting with $1_R$. We also define a recursive transform over $H_R$ that, together to the isomorphism, allows to highlight new relations among some well studied integer sequences. Moreover, these connections allow to introduce a family of polynomials connected to the D'Arcais numbers and the Ramanujan tau function. In this way, we also deduce relations involving the Bell polynomials, the divisor function and the Ramanujan tau function. Finally, we highlight a connection between Cauchy and Dirichlet products.
\end{abstract}



\section{An isomorphism between the convolution product and the componentwise sum}

Given a commutative ring $(R,+,\cdot)$ with identity, let $H_R$ denote the set of all the sequences $\boldsymbol{a}=(a_n)_{n=1}^{\infty}=(a_1,a_2,a_3,...)$, with $a_n \in R$, for all $n \geq 1$. If we consider the operation componentwise sum in $H_R$, denoted by $+$, it is well known that $(H_R, +)$ is a group. Many other interesting operations between elements of $H_R$ can be defined. In the following, we focus on the convolution product (also called Cauchy product), defined by
$$\boldsymbol{a} *\boldsymbol{b} = \boldsymbol{c}, \quad c_{n+1} = \sum_{h=0}^n a_{h+1} b_{n-h+1}, \quad \forall n \geq 0,$$
given any $\boldsymbol{a}, \boldsymbol{b} \in H_R$. It is worth noting that $(H_R, +, *)$ is a ring with identity $\mathbb{1} = (1,0,0,0,...)$. Moreover, let us observe that given any $\boldsymbol{a}, \boldsymbol{b} \in H_R$ with ordinary generating functions $A(t)$ and $B(t)$, respectively, then $\boldsymbol{c} = \boldsymbol{a} * \boldsymbol{b}$ has o.g.f. $A(t)B(t)$.

\begin{remark}
The Cauchy product is strictly related to the binomial convolution product (also called Hurwitz product). Given $\boldsymbol{a}, \boldsymbol{b} \in H_R$, the Hurwitz product is defined as $\boldsymbol{a} \star \boldsymbol{b} = \boldsymbol{c}$, where
$$c_{n+1} = \sum_{h = 0}^n \binom{n}{h}a_{h+1}b_{n-h+1},$$
for all $n \geq 0$.
It is easy to see that the following map
$$\gamma: (H_R,+,*) \rightarrow (H_R,+,\star)$$
defined by $\gamma(\boldsymbol{a}) = \boldsymbol{b}$, where $b_{n+1} = n! a_{n+1}$, for all $n \geq 0$, is an isomorphism. The ring $(H_R,+,\star)$ is isomorphic to the Hurwitz series ring (see, e.g., \cite{Kei1}) that has been extensively studied. Some results on the Hurwitz series ring and on the binomial convolution product can be found in \cite{BCM, BCM2, Ben1, Ben3, Gha, Kei2, Sin, Tot}.
\end{remark}

A sequence $\boldsymbol{a} \in H_R$ is invertible with respect to $*$ if and only if $a_1 \in R$ is invertible with respect to $\cdot$. We denote $H_R^{*_C}$ the set of invertible elements of $H_R$ with respect to $*$. Thus, $(H_R^{*_C}, *)$ is a group whose identity is $(1,0,0,...)$, which is isomorphic to $(\tilde H_R, *)$, where $\tilde H_R$ is the set of sequences in $H_R$ starting with 1. Indeed, the map $f: H_R^{*_C} \rightarrow \tilde H_R$, such that $f(\boldsymbol{a}) = (1, a_1^{-1} \cdot a_2, a_1^{-1} \cdot a_3, ...)$, is injective and surjective. Moreover, since the componentwise multiplication is distributive with respect to $*$, we have that $f$ is an isomorphism. 

We can write an isomorphism $\phi$ between $(H_R, +)$ and $(\tilde H_R, *)$ defined by $\phi(\boldsymbol{a}) = \boldsymbol{u}$, where $\boldsymbol{u}$ is the sequence with o.g.f. 
$$U(t) = \sum_{h=0}^{\infty} u_{h+1}t^h = \prod_{k=1}^{\infty} (1 + t^k)^{a_k}.$$
Remembering that the Cauchy product between two sequences is equivalent to the product between their o.g.f., we clearly have
$$\phi(\boldsymbol{a} + \boldsymbol{b}) = \phi(\boldsymbol{a}) * \phi(\boldsymbol{b}),$$
given any $\boldsymbol{a}, \textbf{b} \in H_R$. We can also observe that the function $1 + t$ is the o.g.f. of the sequence $\boldsymbol{b}^{(1)} = (1,1,0,0,...)$ and $1 + t^k$ is the o.g.f. of the sequence $\boldsymbol{b}^{(k)} = (1,0,...,0,1,0,0,...)$, i.e., the sequence with 
$$b_1^{(k)} = b_{k+1}^{(k)} = 1$$
and 0 otherwise. These sequences are mutually independent, in the sense that $<\boldsymbol{b}^{(i)}> \cap <\boldsymbol{b}^{(j)}> = \{\mathbb 1\}$, with $i \not= j$, where $<\boldsymbol{a}>$ denotes the subgroup generated by $\boldsymbol{a} \in (\tilde H_R,*)$. Thus, we may consider the set $\{\boldsymbol{b}^{(1)}, \boldsymbol{b}^{(2)}, ...\}$ as a basis of $(\tilde H_R, *)$ and in this way it is clear that $\phi$ is an isomorphism.

\begin{proposition}
Given $\boldsymbol{a} \in H_R$, if $\boldsymbol{u} = \phi(\boldsymbol{a})$, then
$$u_{n+1} = \sum_{i_1 + 2i_2 + ... + ni_n = n}\binom{a_1}{i_1}\binom{a_2}{i_2}\cdots\binom{a_n}{i_n}$$
for any $n \geq 1$
\end{proposition}
\begin{proof}
From
$$\prod_{k=1}^{\infty}(1+t^k)^{a_k} = \sum_{i_1=0}^{\infty}\sum_{i_2=0}^{\infty}\cdots\sum_{i_n=0}^{\infty}\cdots\binom{a_1}{i_1}\binom{a_2}{i_2}\cdots\binom{a_n}{i_n}\cdots t^{\sum_{j=1}^{\infty}ji_j}$$
the thesis easily follows.
\end{proof}

\begin{corollary}
Given $\boldsymbol{a} \in H_R$, let $\boldsymbol{u}$ be the sequence having o.g.f. $\prod_{k=1}^{\infty}(1-t^k)^{a_k}$. Then for all $n\geq1$
$$u_{n+1} = \sum_{i_1 + 2i_2 + ... + ni_n = n}(-1)^{i_1+i_2+...+i_n}\binom{a_1}{i_1}\binom{a_2}{i_2}\cdots\binom{a_n}{i_n}$$
\end{corollary}

For the next proposition, we need to recall the definition of the partial ordinary Bell polynomials.

\begin{definition} \label{def:bell}
Let us consider the sequence of variables $X=(x_1,x_2,\ldots)$. The \emph{complete ordinary Bell polynomials} are defined by
$$B_0(X)=1, \quad \forall n\geq 1 \quad B_n(X)=B_n(x_1,x_2,\ldots,x_n)=\sum_{k=1}^nB_{n,k}^0(X),$$
where $B_{n,k}^0(X)$ are the \emph{partial ordinary Bell polynomials},  with
$$ B_{0,0}^0(X)=1, \quad \forall n\geq 1 \quad B_{n,0}^0(X)=0,\quad \forall k\geq 1 \quad B_{0,k}^0(X)=0,$$
$$B_{n,k}^0(X) =B_{n,k}^0(x_1,x_2,\ldots,x_{n}) =k!\sum_{\substack {i_1  + 2i_2  +  \cdots  + ni_n = n \\ i_1  + i_2  +  \cdots  + i_n  = k}} \prod_{j=1}^{n}\frac{x_j^{i_j }}{i_{j}!},$$
or, equivalently,
$$B_{n,k}^0(X)=B_{n,k}^0(x_1,x_2,\ldots,x_{n-k+1})=k!\sum_{\substack {i_1  + 2i_2  +  \cdots  + (n-k+1)i_{n-k+1} = n-k+1 \\ i_1  + i_2  +  \cdots  + i_{n-k+1}  = k}}\prod_{j=1}^{n-k+1}\frac{x_j^{i_j }}{i_{j}!},$$
satisfying the equality
$$ \left(\sum_{n\geq1}x_nz^n\right)^k=\sum_{n \geq k}B_{n,k}^0(X)z^n. $$
\end{definition}

\begin{proposition}
Let $\boldsymbol{a} \in H_R$ and $\boldsymbol{u} \in \tilde H_R$ be two sequences such that $\boldsymbol{u} = \phi(\boldsymbol{a})$. Then $a_{1}=u_{2}$ and
$$
a_{n}=\underset{\begin{array}{c}
1\leq k\leq n-1\\
k|n
\end{array}}{\sum}\frac{k}{n}\left(-1\right)^{\frac{n}{k}}a_{k}+\stackrel[h=1]{n}{\sum}\frac{\left(-1\right)^{h-1}}{h}B_{n,h}^{0}\left(u_{2},u_{3},\ldots,u_{n-h+2}\right)
$$
for any $n \geq 2$.
\end{proposition}
\begin{proof}
Since $u_1 = 1$, we have 
$$\log\left(1+\stackrel[n=1]{\infty}{\sum}u_{n+1}t^{n}\right)=\stackrel[h=1]{\infty}{\sum}\frac{\left(-1\right)^{h-1}}{h}\left(\stackrel[n=1]{\infty}{\sum}u_{n+1}t^{n}\right)^{h}=\stackrel[n=1]{\infty}{\sum}\left(\stackrel[h=1]{n}{\sum}\frac{\left(-1\right)^{h-1}}{h}B_{n,h}^{0}\left(u_{2},u_{3},\ldots,u_{n-h+2}\right)\right)t^{n}.$$
On the other hand, we have
$$
\stackrel[k=1]{\infty}{\sum}a_{k}\left(\log\left(1+t^{k}\right)\right)=\stackrel[k=1]{\infty}{\sum}a_{k}\stackrel[s=1]{\infty}{\sum}\frac{\left(-1\right)^{s-1}}{s}t^{ks}=\stackrel[s=1]{\infty}{\sum}\frac{\left(-1\right)^{s-1}}{s}\stackrel[k=1]{\infty}{\sum}a_{k}t^{sk}.
$$
Let us observe that the exponents of $t$ can be increasingly ordered starting from 1 by posing $sk = n$ and considering as coefficient of $t^n$ the following sum:
$$
\underset{\begin{array}{c}
1\leq k\leq n\\
k|n
\end{array}}{\sum}\frac{k}{n}\left(-1\right)^{\frac{n}{k}-1}a_{k}.
$$
Hence, we get
\begin{equation} \label{eq:10}
\underset{\begin{array}{c}
1\leq k\leq n\\
k|n
\end{array}}{\sum}\frac{k}{n}\left(-1\right)^{\frac{n}{k}-1}a_{k}=\stackrel[h=1]{n}{\sum}\frac{\left(-1\right)^{h-1}}{h}B_{n,h}^{0}\left(u_{2},u_{3},\ldots,u_{n-h+2}\right)
\end{equation}
from which the thesis follows.
\end{proof}

\begin{corollary}
Let $\boldsymbol{a} \in H_R$ and $\boldsymbol{u} \in \tilde H_R$ be two sequences such that the o.g.f. of $\boldsymbol{u}$ is $\prod_{k=1}^{\infty}(1-t^k)^{a_k}$. Then $a_{1}=-u_{2}$ and for all $n\geq2$
\begin{equation} \label{eq:12}
a_{n}=-\underset{\begin{array}{c}
1\leq k\leq n-1\\
k|n
\end{array}}{\sum}\frac{k}{n}a_{k}-\stackrel[h=1]{n}{\sum}\frac{\left(-1\right)^{h-1}}{h}B_{n,h}^{0}\left(u_{2},u_{3},\ldots,u_{n-h+2}\right).
\end{equation}
\end{corollary}

If we exploit the relation between partial ordinary Bell polynomials and ordinary Bell polynomials, we can obtain

$$\stackrel[h=1]{n}{\sum}\frac{\left(-1\right)^{h-1}}{h}B_{n,h}^{0}\left(u_{2},u_{3}\ldots,u_{n-h+2}\right)=\frac{1}{n!}\stackrel[h=1]{n}{\sum}\left(-1\right)^{h-1}\left(h-1\right)!B_{n,h}\left(1!u_{2},2!u_{3},\ldots,\left(n-h+1\right)!u_{n-h+2}\right).$$

We also recall the inversion formula of the partial Bell polynomials, i.e.,

$$x_{n}=\stackrel[h=1]{n}{\sum}\left(-1\right)^{h-1}\left(h-1\right)!B_{n,h}\left(y_{1},y_{2},\ldots,y_{n-h+1}\right)\Longleftrightarrow y_{n}=\stackrel[h=1]{n}{\sum}B_{n,h}\left(x_{1},x_{2},\ldots,x_{n-h+1}\right).$$

Applying the last two equations to \eqref{eq:10} and \eqref{eq:12}, we get for all $n\geq 1$

$$
u_{n+1}=\frac{1}{n!}\stackrel[h=1]{n}{\sum}B_{n,h}\left(\hat{a_{1}},\hat{a_{2}},\ldots,\hat{a}_{n-h+1}\right)\Longleftrightarrow\hat{a_{n}}=n!\underset{\begin{array}{c}
1\leq k\leq n\\
k|n
\end{array}}{\sum}\frac{k}{n}\left(-1\right)^{\frac{n}{k}-1}a_{k}
$$
and
\begin{equation} \label{eq:unbell2}
u_{n+1}=\frac{1}{n!}\stackrel[h=1]{n}{\sum}B_{n,h}\left(\bar{a_{1}},\bar{a}_{2},\ldots,\bar{a}_{n-h+1}\right)\Longleftrightarrow\bar{a}_{n}=-n!\underset{\begin{array}{c}
1\leq k\leq n\\
k|n
\end{array}}{\sum}\frac{k}{n}a_{k}.
\end{equation}

In the next section, we define a transform that together with $\phi$ allows us to introduce an interesting family of polynomials connected to the Ramanujan tau function.

\section{A recursive transform}

\begin{definition} \label{def:autogen}
Given a sequence $\boldsymbol{a} \in H_R$, we define a map $\alpha: H_R \rightarrow H_R$ such that $\alpha(\boldsymbol{a}) = \boldsymbol{b}$ is recursively defined by the following rules:
\begin{enumerate}
\item $b_1 = a_1$, $b_2 = a_2$;
\item for any $i \geq 2$, let $j$ be the first index such that $b_j = a_i$, then $b_{j+k} = b_k$, for $k=1,2,...,j - 1$, and $b_{2j} = a_{i+1}$.
\end{enumerate}
\end{definition}

\begin{example}
Let us consider the sequence $\boldsymbol{a} = (a_1, a_2, a_3, ...)$, by the first rule in Definition \ref{def:autogen} we know that the first two elements of $\boldsymbol{b} = \alpha(\boldsymbol{a})$ are $a_1$ and $a_2$, i.e.,
$$\boldsymbol{b} = (a_1, a_2, ...).$$
Applying the second rule for $i = 2$ we get that $b_3 = a_1$ and $b_4 = a_3$:
$$\boldsymbol{b} = (a_1, a_2, a_1, a_3, ...).$$
Now, considering $i = 3$ (always in the second rule of Definition \ref{def:autogen}) we get that after the first occurrence of $a_3$ in $\boldsymbol{b}$ we have the elements $a_1, a_2, a_1$ that precede $a_3$ itself and after that we will have $a_4$:
$$\boldsymbol{b} = (a_1, a_2, a_1, a_3, a_1, a_2, a_1, a_4...),$$
and so on.
\end{example}

By Definition \ref{def:autogen}, we can write the elements of $\boldsymbol{b} = \alpha(\boldsymbol{a})$ as follows:
$$b_{2^{t - 1}(2 k - 1)} = a_t,$$
for $t = 1, 2, ...$ and any $k \geq 1$.

From now on, we will focus on $R = \mathbb{Z}$. We can see that the transform $\alpha$ and the isomorphism $\phi$ appear to work in an interesting way on integer sequences, highlighting interesting combinatorial aspects. In the next propositions, we prove some new relations among integer sequences by means of $\alpha$ and $\phi$. Moreover, these connections will allow us to introduce in an original way a very interesting family of polynomials connected to the D'Arcais numbers and the Ramanujan tau function.

\begin{proposition}
Let $\boldsymbol{r} = (r_n)_{n=1}^{\infty} = (2, 3, 2, 4, 2, 3, 2, 5, ...)$ be the sequence that counts the number of divisors of $2n$ of the form $2^k$ (sequence A085058 in OEIS \cite{oeis}; note that in OEIS sequences start with index 0 whereas our sequences start with index 1). Given the sequence $\boldsymbol{a}=(2, 3, 4, 5, 6, ...)$, we have
$$\alpha(\boldsymbol{a}) = \boldsymbol{r}.$$
\end{proposition} 
\begin{proof}
Let $\boldsymbol{q} = (q_n)_{n=1}^{\infty} = (1, 2, 1, 3, 1, 2, 1, 4, 1, 2, 1, 3, 1, ...)$ be the sequence of 2--adic valuations of $2n$ (sequence A001511 in OEIS). It is well known that  $\sum_{h=1}^{\infty}q_{h}t^{h}=\sum_{k=0}^{\infty} \cfrac{t^{2^k}}{1 - t^{2^k}}$ and $\boldsymbol{r} = \boldsymbol{q} + (1, 1, 1, ...)$. Thus, 
$$ R(t) =\sum_{h=1}^{\infty}r_{h}t^{h}=\sum_{h=1}^{\infty}(q_{h}+1)t^{h}=\sum_{k=0}^{\infty} \cfrac{t^{2^k}}{1 - t^{2^k}} + \cfrac{t}{1 - t}. $$ 
Let $G_n(t)=\sum_{h=1}^{2^{n}}\alpha(\boldsymbol{a})[h]t^h$ where, for $h\geq 1$, $\alpha(\boldsymbol{a})[h]$ is the $h$--th element of $\alpha(\boldsymbol{a})$. We have 
\begin{itemize}
\item $G_0(t) = 2t$
\item $G_1(t) = 2t + 3t^2 = (1 + t) G_0(t) + t^2$
\item $G_2(t) = 2t +3t^2+2t^3+4t^4=(1 + t^2) G_1(t) + t^4$.
\end{itemize}
In general, remembering that $\boldsymbol{a}=(2, 3, 4, 5, 6, ...)$ and by definition of $\alpha$, it is easy to check that
$$G_n(t) = (1 + t^{2^{n-1}}) G_{n-1}(t) + t^{2^n}.$$
From this, we get
$$G_{n+1}(t) - G_n(t^2) = (1 + t^{2^n}) (G_n(t) - G_{n-1}(t^2))$$
from which
$$G_{n+1}(t) - G_n(t^2) = (1 + t^{2^n})(1 + t^{2^{n-1}})...(1 + t^2)(G_1(t) - G_0(t^2)) = \cfrac{2t+t^2}{1 - t^2}(1 - t^{2^{n+1}}).$$
Now, if we consider $n \rightarrow + \infty$ (for $\lvert t \rvert < 1$), we have
$$G(t) - G(t^2) = \frac{2t+t^2}{1 - t^2}=\frac{t}{1-t}+\frac{1}{1-t}-\frac{1}{1-t^2}.$$
Thus, given the following differences
$$\begin{cases} G(t) - G(t^2) = \cfrac{t}{1 - t} + \cfrac{1}{1 - t} - \cfrac{1}{1 - t^2} \cr G(t^2) - G(t^4) = \cfrac{t^2}{1 - t^2} + \cfrac{1}{1 - t^2} - \cfrac{1}{1 - t^4} \cr ... \cr G(t^{2^n}) - G(t^{2^{n+1}}) = \cfrac{t^{2^n}}{1 - t^{2^n}} + \cfrac{1}{1 - t^{2^n}} - \cfrac{1}{1 - t^{2^{n+1}}} \end{cases}$$
and finally we obtain
$$ G(t) - G(t^{2^{n+1}}) = \sum_{k=0}^{\infty}\cfrac{t^{2^k}}{1 - t^{2^k}} + \cfrac{1}{1 - t} - \cfrac{1}{1 - t^{2^{n+1}}}.$$
For $n \rightarrow \infty$ (for $\lvert t \rvert < 1$), we have $G(t) = R(t)$ (note that $G(0) = 0$).
\end{proof}

\begin{corollary}\label{corn}
With the notation of the previous proposition, if $\boldsymbol{n}=(1,2,3,\dots)$ is the sequence of positive integers, we have $\alpha(\boldsymbol{n}) = \boldsymbol{q}$.
\end{corollary}

\begin{proposition} \label{prop:fa}
With the notation of the previous proposition, we have $\phi(\boldsymbol{q}) = \boldsymbol{p}$, where $\boldsymbol{p} = (p_n)_{n=1}^{\infty} = (1, 1, 2, 3, 5, 7, ...)$ counts the number of partitions of $n - 1$ (sequence A000041 in OEIS).
\end{proposition}
\begin{proof}
The elements of the sequence $\boldsymbol{q}$ satisfy the following conditions:
$$\begin{cases} q_{2n - 1} = 1 \cr q_{2n} = q_n + 1 \end{cases}$$
for all $n \geq 1$. The o.g.f. of $\phi(\boldsymbol{q})$ is
$$f(t) = \prod_{k = 1}^{\infty}(1 + t^k)^{q_k} = \prod_{s = 1}^{\infty} (1 + t^{2s})^{q_{2s}}\prod_{s = 1}^{\infty} (1 + t^{2s - 1})^{q_{2s-1}}=\prod_{k=1}^{\infty}(1+t^k) f(t^2),$$
from which we get
$$\frac{f(t^{2^{i-1}})}{f(t^{2^i})} = \prod_{k=1}^{\infty}(1 + t^{2^{i-1}k}).$$
Noting that
$$\cfrac{f(t)}{f(t^{2^n})} = \prod_{i=1}^n\frac{f(t^{2^{i-1}})}{f(t^{2^i})} = \prod_{k=1}^{\infty}\left( \prod_{i=1}^n(1 + t^{2^{i-1}k}) \right),$$
we obtain
$$\cfrac{f(t)}{f(t^{2^n})} = \prod_{k=1}^{\infty}\cfrac{1-t^{2^nk}}{1-t^k}.$$
Now, for $n \rightarrow \infty$ (with $\lvert t \rvert < 1$), we have $f(t) = \prod_{k=1}^{\infty}\cfrac{1}{1-t^k}$ that is the o.g.f. of $\boldsymbol p$.

\end{proof}

\begin{proposition}
With the notation of the previous propositions, we have that $\phi(\boldsymbol{r}) = \boldsymbol{\tilde{p}}$, where $\boldsymbol{\tilde{p}} = (\tilde{p}_n)_{n=1}^{\infty} = (1, 2, 4, 8, 14, ...)$ is the sequence that counts the number of partitions of $n$ where there are two kinds of odd parts (see sequence A015128 in OEIS).
\end{proposition}
\begin{proof}
Since $\boldsymbol{p} = \phi(\boldsymbol{q})$ and $\boldsymbol{r} = \boldsymbol{q} + (1, 1, 1, ...)$, we have
$$\prod_{k=1}^{\infty}\cfrac{1}{1-t^k} = \prod_{k=1}^{\infty}(1+t^k)^{q_k}.$$
Thus, the o.g.f. of $\phi(\boldsymbol{r})$ is
$$\prod_{k=1}^{\infty}(1+t^k)^{r_k} = \prod_{k=1}^{\infty}(1+t^k)^{q_k+1}=\prod_{k=1}^{\infty}(1+t^k)\prod_{k=1}^{\infty}(1+t^k)^{q_k}=\prod_{k=1}^{\infty}\cfrac{1+t^k}{1-t^k},$$
that is the o.g.f. of $\boldsymbol{\tilde{p}}$ (see, e.g., sequence A015281 in OEIS \cite{oeis}).
\end{proof} 

\section{The D'Arcais numbers and the Ramanujan tau function }

As a consequence of Corollary \ref{corn} and  Proposition \ref{prop:fa}, we have  that $\boldsymbol{p}=\phi(\boldsymbol{q})=\phi(\alpha(\boldsymbol{n}))$ has o.g.f. $\prod_{k=1}^{\infty}\cfrac{1}{1 - t^k}$. Thus introducing  a variable $x$, the o.g.f. of the polynomial sequence $\phi(\alpha(x\boldsymbol{n}))=\phi(\alpha(x, 2x, 3x, ...))=(\varphi_{n}(x))_{n=1}^{\infty}$ clearly is
$$d(t) = \prod_{k=1}^{\infty} \cfrac{1}{(1 - t^k)^x}.$$

\begin{example}
For $n = 1, 2, 3, 4, 5$, the polynomials above introduced are
$$\varphi_1(x) = 1$$
$$\varphi_2(x) = x$$
$$\varphi_3(x) = \frac{1}{2}x(x + 3)$$
$$\varphi_4(x) = \frac{1}{6}x(x^2 + 9x + 8)$$
$$\varphi_5(x) = \frac{1}{24}x(x^3 + 18x^2 + 59x + 42).$$
\end{example}

The coefficients of the polynomials $\varphi_n(x)$ are given by the D'Arcais numbers (\cite{Comtet}, pag. 159). These polynomials seem to be very interesting, since many important sequences appear as special cases. For instance, we have seen that $\varphi_n(1)_{n=1}^{\infty}$ is the fundamental sequence that counts the number of partitions of an integer $n$ (sequence A000041 in OEIS). Moreover, $\varphi_n(-1)_{n=1}^{\infty}$ is the sequence that counts the number of different partitions of an integer $n$ into parts of -1 different kinds (sequence A010815 in OEIS). A very famous sequence belonging to the family of polynomials $\varphi_n(x)$ is the sequence of the Ramanujan numbers (A000594 in OEIS), i.e., the Ramanujan tau function introduced in \cite{Rama}, for some recent studies see, e.g., \cite{Charles}, \cite{Lygeros}, \cite{Baruah}.
The Ramanujan tau function corresponds to the sequence $(\varphi_n(-24))_{n=1}^{\infty}$. Thus, we have seen that it can be introduced in a combinatorial way as $\phi(\alpha(-24\boldsymbol{n}))$. Hence, the properties of the polynomials $\varphi_n(x)$ can be applied, in particular, to the Ramanujan tau function.

\begin{remark}
Ramanujan conjectured the multiplicative property of the tau function that was proved by Mordell \cite{Mor} by using modular functions. However, an elementary proof of this property still misses. Here, we would like to emphasize that an elementary proof of this property could be reached by proving that the polynomial $\varphi_{mn}(x) - \varphi_m(x)\varphi_n(x)$ has -24 as root when $m, n$ are coprime.
\end{remark}
 
In the following proposition, we summarize some interesting relations involving the Bell polynomials, the sum of divisors function and the polynomials $\varphi_n(x)$. 

\begin{proposition}
Let us consider the sum of divisors function $\sigma(m) = \sum_{d|m}d$. Then for all $n\geq1$
\begin{itemize}
\item $\varphi_{n+1}(x) = \cfrac{1}{n!} \stackrel[h=1]{n}{\sum}B_{n,h}\left(0!\sigma\left(1\right),1!\sigma\left(2\right),\ldots,\left(n-h\right)!\sigma\left(n-h+1\right)\right)x^{h}$
\item $\varphi_{n+1}(x) = \stackrel[h=1]{n}{\sum}\frac{1}{h!}B_{n,h}^{0}\left(\sigma\left(1\right),\frac{\sigma\left(2\right)}{2},\ldots,\frac{\sigma\left(n-h+1\right)}{n-h+1}\right)x^{h}$
\item $\sum_{h=0}^{\infty}\varphi_{h+1}(x)t^h = e^{x\sum_{k=1}^{\infty}\frac{\sigma(k)}{k}t^k}$
\end{itemize}
\end{proposition}
\begin{proof}
Let us consider the sequence $a = (-x)_{n=1}^{\infty}$. In this case the sequence $(\bar a_n)_{n=1}^{\infty}$ defined in  equation \eqref{eq:unbell2} is
$$\bar a_n = x n! \underset{\begin{array}{c}
1\leq k\leq n\\
k|n
\end{array}}{\sum}\frac{k}{n}=x\left(n-1\right)!\sigma\left(n\right).$$
Thus, from equation \eqref{eq:unbell2} we have
$$\varphi_{n+1}(x) = \cfrac{1}{n!} \stackrel[h=1]{n}{\sum}B_{n,h}\left(0!\sigma\left(1\right),1!\sigma\left(2\right),\ldots,\left(n-h\right)!\sigma\left(n-h+1\right)\right)x^{h}$$
for any $n \geq 1$, where we have also used the homogeneity property of the Bell polynomials. In terms of partial Bell polynomials we can write
\begin{equation}\label{eq:rama-bell} \varphi_{n+1}(x) = \stackrel[h=1]{n}{\sum}\frac{1}{h!}B_{n,h}^{0}\left(\sigma\left(1\right),\frac{\sigma\left(2\right)}{2},\ldots,\frac{\sigma\left(n-h+1\right)}{n-h+1}\right)x^{h}. \end{equation}
Since
$$\exp\left(\underset{n\geq1}{\sum}x_{m}z^{m}\right)=\underset{n\geq0}{\sum}\left(\stackrel[k=0]{n}{\sum}\frac{1}{k!}B_{n,k}^{0}\left(x_{1},x_{2},\ldots,x_{n-k+1}\right)\right)z^{n},$$
by equation \eqref{eq:rama-bell} we can also observe that
$$\stackrel[h=0]{\infty}{\sum}\varphi_{h+1}\left(x\right)t^{h}=e^{x\stackrel[k=1]{\infty}{\sum}\frac{\sigma\left(k\right)}{k}t^{k}}.$$
\end{proof}

\section{A connection between Cauchy and Dirichlet products}
The Dirichlet convolution $\odot$ is a well known product between arithmetic functions that can be also defined between two sequences $\boldsymbol{a}$ and $\boldsymbol{b}$ as follows:
$$\boldsymbol{a} \odot \boldsymbol{b} = \boldsymbol{c}, \quad c_n := \sum_{d|n} a_d b_{n/d}, \quad \forall n \geq 1$$
In the following, we deal with the set of sequences starting with 1, i.e., the set $\tilde H_R$ and we define a transform $F: (\tilde H_R, *) \rightarrow (\tilde H_R, \odot)$.

Given the sequence of prime numbers $(p_1, p_2, ...)$, we define $\nu_i(n)$ the $p_i$--adic evaluation of the positive integer $n$, i.e., $\nu_i(n) = e_i$, where $e_i$ is the greatest nonnegative integer such that $p_i^{e_i}$ divides $n$. We define the transform $F$ as follows:
\begin{equation} \label{eq:dirichlet} F(\boldsymbol{a}) = \boldsymbol{b}, \quad b_n := \prod_{i=1}^{\infty} a_{\nu_i(n) + 1}, \quad \forall n \geq 1. \end{equation}
In the following, for the sake of simplicity, we also use the following notation:
$$F(a_n) = b_n,$$
where $b_n$ is given by \eqref{eq:dirichlet}.

By definition, given $\gcd(m,n) = 1$, we have
$$F(a_{mn}) = F(a_m) F(a_n).$$
Hence, if $n = \prod_{i=1}^{\infty} p_i^{e_i}$ (where, some $e_i$ can possibly be equal to zero), then
$$F(a_n) = \prod_{i=1}^{\infty}F(a_{p_i^{e_i}}).$$
Now, we would like to prove that given $\boldsymbol{a}, \boldsymbol{b} \in \tilde H_R$, we have $F(\boldsymbol{a} * \boldsymbol{b}) = F(\boldsymbol{a})\odot F(\boldsymbol{b})$. It is sufficient to prove that for any $i \geq 1$ and $m \geq 0$, we have
$$F(r_{p_i^m}) = s_{p_i^m},$$
where $\boldsymbol{r} = \boldsymbol{a} * \boldsymbol{b}$ and $\boldsymbol{s} = F(\boldsymbol{a})\odot F(\boldsymbol{b})$. Considering that $F(a_{p_i^t}) = a_{t+1}$, we have
$$s_{p_i^m} = \sum_{d|p_i^{m}}F(a_d)F(b_{p_i^m|d}) = \sum_{t=0}^mF(a_{p_i^t})F(b_{p_i^{m-t}}) = \sum_{t=0}^ma_{t+1}b_{m-t+1}=r_{m+1}=F(r_{p_i^m}).$$
If $F(a_n) = 0$ for any $n > 0$, then $F(a_p) = F(a_p^2) = ... = F(a_p^m) = ... = 0$, which implies $a_2 = ... = a_m = ... = 0$. Thus, $ker(F)$ only contains the identity and $F$ is a monomorphism.

\end{document}